\documentclass[11pt]{article}

\usepackage{latexsym}
\usepackage{amsmath}
\usepackage{amsfonts}
\usepackage{amssymb}
\usepackage{amsthm}
\usepackage{psfrag}
\usepackage{epsfig}
\usepackage{color}

\newcommand{\rreg}{{r_{\rm reg}}}
\newcommand{\riso}{{r_{\rm iso}}}
\newcommand{\rcomp}{{r_{\rm c}}}
\newcommand{\rcompinv}{{r_{\rm c}^{-1}}}
\newcommand{\rcompinvt}{{r_{\rm c}^{-2}}}

\newcommand{\R}{\mathbb{R}}

\newcommand{\fR}{\mbox{\footnotesize\rm I\kern-.18em R}}
\newcommand{\sR}{\mbox{\small\rm I\kern-.18em R}}
\newcommand{\N}{\mbox{\rm I\kern-.18em N}}
\newcommand{\dist}{\mathop{\rm dist}\nolimits}

\newcommand{\<}{\langle}
\renewcommand{\>}{\rangle}
\newcommand{\jump}[1]{[#1]}
\newcommand{\curl}{\mathop{\rm curl}\nolimits}
\newcommand{\bcurl}{\mathop{\rm{\bf curl}}\nolimits}

\newcommand{\curlS}[1]{\mathop{\rm curl_{\rm{#1}}}\nolimits}

\newcommand{\bcurlS}[1]{\mathop{\rm{\bf curl}_{\rm{#1}}}\nolimits}

\newcommand{\curlT}{\curlS{\mathcal{T}}}
\newcommand{\bcurlT}{\bcurlS{\mathcal{T}}}

\newcommand{\G}{{\Gamma}}

\newcommand{\CT}{{\cal T}}

\newcommand{\bn}{{\bf n}}
\newcommand{\bt}{{\bf t}}
\newcommand{\bvarphi}{{\mbox{\boldmath $\varphi$}}}

\newcommand{\tH}{\tilde H}

\newcommand{\bV}[1]{\mbox{\boldmath $V$}\!\!_{#1}}

\newtheorem{theorem}{Theorem}

\newtheorem{remark}[theorem]{Remark}
\newtheorem{lemma}[theorem]{Lemma}
\newtheorem{corollary}[theorem]{Corollary}

\title{
A non-conforming domain decomposition approximation for the
Helmholtz screen problem with hypersingular operator
\thanks{Supported by CONICYT through FONDECYT project 1150056 and Anillo ACT1118 (ANANUM).}
}

\author{
Norbert Heuer
\thanks{
Facultad de Matem\'aticas, Pontificia Universidad Cat\'olica de Chile,
Avenida Vicu\~na Mackenna 4860, Macul, Santiago, Chile.
email: {\tt \{nheuer,gjsalmeron\}@mat.puc.cl}}
\and
Gredy Salmer\'on~$^\dagger$
}

\begin{document}
\date{}
\maketitle

\bigskip
\begin{abstract}
We present and analyze a non-conforming domain decomposition approximation
for a hypersingular operator governed by the Helmholtz equation in three dimensions.
This operator appears when considering the corresponding Neumann problem in
unbounded domains exterior to open surfaces. We consider small wave numbers
and low-order approximations with Nitsche coupling across interfaces.
Under appropriate assumptions on mapping properties of the weakly singular
and hypersingular operators with Helmholtz kernel, we prove that this method
converges almost quasi-optimally. Numerical experiments confirm our error estimate.


\end{abstract}

\section{Introduction}

In recent years we have started to develop non-conforming boundary elements, in the
sense that approximations to boundary integral equations with hypersingular
operators can be discontinuous. Approaches consider both element-wise discontinuous
methods \cite{HeuerS_09_CRB,HeuerM_13_DGB} and domain decomposition techniques,
mortar coupling in \cite{HealeyH_10_MBE}
and Nitsche coupling in \cite{ChoulyH_12_NDD}. However, all results are
restricted to the simple model problem of the Laplacian.

In this paper we extend the Nitsche domain decomposition method from
\cite{ChoulyH_12_NDD} to the hypersingular operator $W_k$ stemming from the
Helmholtz problem with small wave number $k$. Traditional variational analysis of
this operator is based on the theory of Fredholm operators since, for small wave numbers,
$W_k$ can be handled as a compact perturbation of the elliptic operator $W_0$
which corresponds to the Laplace case.
This approach is not applicable to our non-conforming discrete setting.
The energy space of $W_0$, e.g. defined on an open surface $\G$,
is a trace space of $H^1(\Omega)$ (with $\Omega:=\R^3\setminus\bar\G$) and thus
of order $1/2$. In such a space there is no well-defined trace operator.
On the other hand, the analysis of discontinuous approximations requires
the consideration of jumps and thus, traces. Because of this conflict,
numerical analysis of discontinuous approximations of hypersingular integral
equations has been carried out exclusively on the discrete level where traces
are defined as restrictions. In this way arguments from variational settings
can be avoided.
Now, standard numerical analysis of Fredholm operators is based on
compactness arguments which, by nature, are connected with non-discrete
variational settings, cf., e.g.,
\cite{Stephan_87_BIE,HolmMS_96_hpB,MaischakMS_97_AMB,EnglederS_08_SBE}
where the analysis of boundary elements is based on G\r{a}rding's inequality.
In this paper, we present an analysis of the Helmholtz case which reconciles
both seemingly conflicting approaches, the restriction to discrete spaces
and appropriate extension to consider Fredholm operators.
This latter extension is done by providing discrete variants of
a G\r{a}rding's inequality.
Nevertheless, our main result will be based on three assumptions on the weakly
singular and hypersingular operators whose verification goes beyond the scope
of this paper.

Our analysis also uses a compactness argument. Corresponding estimates generate
unknown constants which depend in most cases on the geometry and possibly other data;
in our case they depend on the order of Sobolev norms.
For this reason, final estimates are based on Sobolev regularities $s>1/2$. Limits of
$s$ tending to $1/2$ cannot be considered since the dependence of the constants
on $s$ is unknown. This is different in the Laplace case where estimates
involving natural norms of order $1/2$ can be established by limits. In this
way quasi-optimal error estimates with poly-logarithmic perturbations appear,
cf., e.g., \cite[Theorem~3.1]{ChoulyH_12_NDD}.
In the Helmholtz case considered here, estimates are less specific by
assuming that Sobolev orders in upper bounds are strictly larger than $1/2$.

Let us note a further complication of discontinuous boundary elements. Discontinuous
(DG) finite elements are usually analyzed considering specific DG-type norms,
comprising broken semi-norms and scaled jump terms. They are tuned to harmonize
with DG-bilinear forms and have also been considered in the boundary element settings
studied in \cite{HeuerS_09_CRB,ChoulyH_12_NDD,HeuerM_13_DGB}. However, in a boundary
integral operator approach one has to consider a post-processing step consisting in
evaluating the underlying representation, e.g.,
$U_h(x)=\mathrm{op}_\G(K,u_h)$ for $x\in\Omega$
($u_h$ denoting the boundary element approximation, and
$\mathrm{op}_\G$ the integral operator with kernel $K$ used for representing the
solution $U=\mathrm{op}_\G(K,u)$ to the original boundary value problem).
One establishes convergence orders for the point-wise evaluation
by applying duality estimates to the integral operator,
\begin{equation} \label{point}
   |U(x)-U_h(x)|=|\mathrm{op}_\G(K,u-u_h)|\le \|u-u_h\|_* \|K(\cdot-x)\|_{*'}.
\end{equation}
Here, $\|\cdot\|_*$ and $\|\cdot\|_{*'}$ denote, respectively, the norm considered to
bound the boundary element error $u-u_h$ and its dual norm, and one uses that
the kernel $K(y-x)$ is smooth for $x\not=y$ ($x\in\Omega$, $y\in\G$).
It is not straightforward to analyze such a duality estimate for a DG-norm.
In this paper we provide an error estimate for a standard Sobolev norm (it is
a broken $H^{1/2}$-norm) so that its dual norm is known and can be used to make
the error estimate \eqref{point} explicit by specifying both norms.

The remainder of this paper is organized as follows. In the next section
we briefly recall some Sobolev norms and present the model problem. We also
formulate two assumptions on which our subsequent analysis is based.
In Section~\ref{sec_discrete} we present the non-conforming domain decomposition
setting and formulate the main result (Theorem~\ref{thm_Cea}), a C\'ea-type estimate.
The following Corollary~\ref{cor_error} establishes the convergence order
of the method. A proof of Theorem~\ref{thm_Cea} is given at the end of
Section~\ref{sec_proofs}, after collecting a number of preliminary results,
including consistency of the discrete method (Lemma~\ref{la_consistent}),
boundedness of the sesquilinear form in broken Sobolev spaces of order $s>1/2$
(Lemma~\ref{la_cont}), discrete G\r{a}rding's inequalities
(Lemma~\ref{la_ell} and Corollary~\ref{cor_ell}), and a lower-order error estimate
based on the Aubin-Nitsche trick (Lemma~\ref{la_AN}). Some numerical experiments
that confirm our estimates are reported in Section~\ref{sec_num}.

Throughout the article, we will use the symbols "$\lesssim$" and "$\gtrsim$"
in the usual sense. In short $a_h(v) \lesssim b_h(v)$ when there exists a constant
$C > 0$ independent of $v$ and the mesh size $h$, such that
$a_h(v)\le C\,b_h(v)$. Also, $a_h(v)\simeq b_h(v)$ means that
$a_h(v)\lesssim b_h(v)$ and $a_h(v)\gtrsim b_h(v)$.

\section{Sobolev spaces and model problem} \label{sec_model}

For $\Omega\subset\R^n$ and $0<s<1$ we define
\[
   \|u\|^2_{H^s(\Omega)}:=\|u\|^2_{L^2(\Omega)} + |u|^2_{H^s(\Omega)}
\]
with semi-norm
\begin{equation} \label{Hs}
    |u|_{H^s(\Omega)} := 
    \Bigl(
    \int_\Omega \int_\Omega \frac{|u(x)-u(y)|^2}{|x-y|^{2s+n}} \,dx\,dy
    \Bigr)^{1/2}.
\end{equation}
For a Lipschitz domain $\Omega$ and $0<s<1$, the space
$\tilde H^s(\Omega)$ is defined as the completion of $C_0^\infty(\Omega)$
under the norm
\[
   \|u\|_{\tilde H^s(\Omega)}
   :=
   \Bigl(
   |u|^2_{H^{s}(\Omega)}
   +
   \int_\Omega \frac{|u(x)|^2}{\dist(x,\partial\Omega)^{2s}} \,dx
   \Bigr)^{1/2}.
\]
For $s\in (0,1/2)$, $\|\cdot\|_{\tilde H^s(\Omega)}$ and $\|\cdot\|_{H^s(\Omega)}$
are equivalent norms whereas for $s\in(1/2,1)$ there holds
$\tilde H^s(\Omega) = H_0^s(\Omega)$, the latter space being the completion
of $C_0^\infty(\Omega)$ with norm in $H^s(\Omega)$.
For $s>0$ the spaces $H^{-s}(\Omega)$ and $\tilde H^{-s}(\Omega)$ are the
dual spaces (with $L^2(\Omega)$ as pivot space)
of $\tilde H^s(\Omega)$ and $H^s(\Omega)$, respectively.
For more details on Sobolev spaces we refer to \cite{LionsMagenes, Grisvard_85_EPN}.

In the following, let $\G$ be a piecewise plane Lipschitz surface.
For simplicity we assume that $\G$ is open with polygonal boundary $\partial\G$.
Sobolev spaces on faces of $\G$ are defined as previously, identifying faces with
sub-domains of $\R^2$, i.e., $n=2$ in \eqref{Hs}. For a closed surface
$\tilde\G$ being the boundary of $\tilde\Omega$ and containing $\G$,
$H^s(\tilde\G)$ is the trace of $H^{s+1/2}(\tilde\Omega)$ ($s>1/2$) and
$\tH^s(\G)$ is the space of functions from $H^s(\tilde\G)$ with support on $\G$.
Dualities with spaces of negative order are defined as previously.
Furthermore, throughout the paper, we use the same notation for Sobolev spaces of
vector-valued functions, taking respective norms component-wise.

Our model problem is:
{\em For given wave number $k>0$ and sufficiently smooth function $f$ find
$u\in\tH^{1/2}(\G)$ such that}
\begin{equation} \label{IE}
   W_ku(x):=-\frac 1{4\pi}\frac {\partial}{\partial \bn_x}
              \int_\G u(y) \frac {\partial}{\partial \bn_y} \frac {e^{ik|x-y|}}{|x-y|}
              \,dS_y
   = f(x),\quad x\in\G.
\end{equation}
Here, $\bn$ is a normal unit vector on $\G$ pointing to one side.

\begin{remark}
The assumption that $\G$ is an open surface implies that \eqref{IE} has a unique solution
for any $k$ with $Im(k)\ge 0$, cf.~\cite{Stephan_87_BIE}. For closed surfaces,
different boundary integral equations like the Burton-Miller formulation are in order,
see~\cite{BurtonM_71_AIE}.
\end{remark}

\begin{remark} \label{rem2}
(i) In the case of the Laplacian, i.e., $k=0$, it is well known that the solution of \eqref{IE}
with appropriate (and sufficiently smooth) right-hand side $f$
(so that it relates to a Neumann Laplace problem)
satisfies $u\in\tH^{r}(\G)$ for any $r<1$, see \cite{vonPetersdorffS_90_RMB,Dauge_88_EBV}.
In the Helmholtz case ($k>0$) Stephan used the theory of pseudo-differential operators to show 
that on open surfaces with smooth boundary curve, and $f\in H^1(\G)$, $u$ has a square-root
edge singularity and that $u\in\tH^r(\G)$ for any $r<1$. We do not know of a specific
analysis on open or closed polyhedral surfaces.\\
(ii) A direct formulation of the Helmholtz problem in $\R^3\setminus\bar\G$ with Neumann
boundary condition satisfies the Sommerfeld radiation condition, i.e., it only considers
outgoing waves. The boundary integral equation with hypersingular operator $W_k$ for the
wave number $k$ reflects this behavior. Changing the sign of $k$ turns the problem into
the non-physical one of incoming waves. In our analysis we will need the adjoint operator of
$W_k$. It can be immediately seen that this is $W_{-k}$, when considering the
$L^2(\G)$-sesquilinear form. Therefore, mapping properties of $W_{-k}$ can be proved
analogously to the ones of $W_k$ by replacing the Sommerfeld radiation condition of
outgoing waves by the one representing incoming waves, cf., e.g.,
\cite{McLean_00_SES} and see also \cite[Remark 3.9.6]{SauterS_11_BEM}.
However, for the particular case of an open polyhedral surface the literature is scarce,
as most specific results concern the Laplacian.
\end{remark}

Considering the two previous remarks, we are making the following assumptions. 

\medskip
\noindent{\bf Assumption 1.}
\emph{There exists $\rreg\in (1/2,1)$ such that the solution $u$ of \eqref{IE}
satisfies $u\in\tH^{\rreg}(\G)$.}

\medskip
\noindent{\bf Assumption 2.}
\emph{There exists $\riso\in (1/2,1)$ such that, for $k>0$, the operator
$W_{-k}:\;\tilde H^\riso(\G)\to H^{\riso-1}(\G)$ is an isomorphism.}

\medskip
A variational formulation of (\ref{IE}) is:
{\em Find $u\in\tilde H^{1/2}(\G)$ such that}
\begin{equation} \label{weak_org}
   \<W_ku, v\>_\G
   = \<f, v\>_\G\quad\forall v\in \tH^{1/2}(\G).
\end{equation}
Here, $\<\cdot,\cdot\>_\G$ denotes the duality pairing 
between $H^{-1/2}(\G)$ and $\tH^{1/2}(\G)$. 
Throughout, this generic notation will be used for the $L^2$-inner product
and other dualities, and the domain is indicated by the index.

A standard boundary element method for the approximate solution of (\ref{weak_org})
is to select a piecewise polynomial subspace $\tilde H_h\subset \tilde H^{1/2}(\G)$
and to define an approximant $\tilde u_h\in\tilde H_h$ by
\[
\<W_k\tilde u_h, v\>_\G
   = \<f, v\>_\G\quad\forall v \in\tilde H_h.
\]

\section{Domain decomposition with Nitsche coupling} \label{sec_discrete}

In this section, we introduce the Nitsche-based boundary element method
for the approximate solution of problem \eqref{weak_org}, and present the
main result, Theorem~\ref{thm_Cea}.

\subsection{Preliminaries} \label{sec_prelim}

We consider a decomposition of $\G$,
\[
   \CT := \{\G_j;\; j=1,\ldots,J\},
\]
where we assume that elements of $\CT$ are plane polygonal surfaces.
Throughout the paper, we will use the notation $v_j$
for the restriction of a function $v$ to a sub-surface $\G_j$ (also called sub-domain).
The decomposition of $\G$ induces product Sobolev spaces of complex-valued functions, e.g.,
\[
   H^s (\CT) := \Pi_j H^s(\G_j)
\]
with corresponding broken semi-norm $|\cdot|_{H^s(\CT)}$, using on each sub-domain the
Sobolev-Slobodeckij semi-norm previously defined.
This notation with decomposition $\CT$ will be used generically, i.e.,
also for the piecewise $L^2$-sesquilinear form
\[
   \< v, w\>_\CT := \sum_j \<v_j, w_j\>_{\G_j}
\]
and its extension by duality to $\tilde H^s(\CT)\times H^{-s}(\CT)$.

We also make use of the surface differential operators $\bcurl$ and $\curl$.
On a subset of $\R^2\times\{0\}$ they amount to
\(
   \bcurl \varphi := \bigl(\partial_{x_2}\varphi, -\partial_{x_1}\varphi, 0\bigr)^T
\)
and
\(
   \curl \bvarphi
   := \partial_{x_1}\varphi_2 - \partial_{x_2}\varphi_1
\)
for sufficiently smooth scalar and vector functions $\varphi$ and
$\bvarphi=(\varphi_1,\varphi_2,\varphi_3)^T$, respectively.
For a definition and analysis on Lipschitz surfaces we refer to \cite{BuffaCS_02_THL}.
The restrictions of these operators to a face $\G_j$ will be denoted by $\bcurl_j$ and $\curl_j$.
Corresponding to the decomposition $\CT$ we also define the broken or piecewise operators
$\bcurlT$ and $\curlT$, e.g.,
\(
   (\bcurlT \varphi)|_{\G_j} := \bcurl_j \varphi_j
\)
($j=1,\ldots,J$) and similarly the other operator.

Let $\gamma$ denote the skeleton of $\CT$, including $\partial\G$.
The jump $\jump{\cdot}$ of functions across $\gamma$ is defined so that it is
compatible with a tangential direction on $\gamma$, appearing when integrating by parts the
surface differential operators. 
More precisely, for a scalar function $v$ (sufficiently $\CT$-piecewise smooth) and a
tangential vector field $\bvarphi$ (sufficiently smooth so that its trace on $\gamma$
is well defined) we define tangential components $\bt(\bvarphi)$ and jumps $\jump{v}$
being compatible with the integration-by-parts formula
\begin{equation} \label{def_t}
   \<\bt(\bvarphi), \jump{v}\>_\gamma
   =
   \<\curlS{\CT}\bvarphi, v\>_\CT
   - \<\bcurlT v, \bvarphi\>_\CT.
\end{equation}
We select a unique tangential direction on $\gamma\setminus\partial\G$
(this fixes the directions of the jumps), and on $\partial\G$ so that $\jump{v}|_{\partial\G}$
is the trace of $v$ on $\partial\G$.

Now, for $s\in [1/2,1]$ and $\nu>0$, we introduce the norm
\[
   \|v\|_{H^s_\nu(\CT)} := \Bigl(|v|_{H^s(\CT)}^2  + \nu \|\jump{v}\|_{L^2(\gamma)}^2\Bigr)^{1/2}.
\]
For $s>1/2$, this is a norm in $H^s(\CT)$ and in the case $s=1/2$, this norm will be
used only for discrete functions whose jumps across $\gamma$ are well defined as elements
of $L^2(\gamma)$.

We end this section with recalling a relation that connects the hypersingular operator $W_k$
with the single layer operator $V_k$ defined by
\[
   V_k\varphi(x) := \frac 1{4\pi} \int_\G \varphi(y) \frac {e^{ik|x-y|}}{|x-y|}\,dS_y,
   \quad \varphi\in \tilde H^{-1/2}(\G),\ x\in\G.
\]
When applied component-wise to vector-valued functions we use the bold face symbol $\bV{k}$.
The operators $W_k$ and $V_k$ satisfy the relation
\begin{equation} \label{WV}
 \<W_k u , v\>_\G = \<\bV{k}\bcurl u, \bcurl v \>_\G - k^2 \<\bV{k} \bn\, u, \bn\, v\>_\G
  \quad \forall u,v \in \tH^{1/2}(\G),
\end{equation}
see \cite{Maue_49_FAB,Nedelec_82_IEN}.
As in previous publications on the Laplacian,
this formula will give rise to our non-conforming discrete formulation of the hypersingular operator.

\subsection{Discrete method and main result}

On every sub-domain $\G_j$ we consider regular, quasi-uniform meshes $\CT_j$, $j=1,\ldots, J$,
of shape-regular elements (quadrilaterals or triangles), $\bar\G_j = \cup_{K\in\CT_j} \bar K$.
The maximum, respectively minimum, diameter of the elements of $\CT_j$ is denoted by
$h_j$, respectively $\underline{h}_j$. We also define
\[
   h := \max\{h_1,\ldots, h_J\},\qquad
   \underline{h} := \min\{\underline{h}_1,\ldots, \underline{h}_J\}.
\]
Throughout this paper we assume that $0<\underline{h}\le h\le C<\infty$.
Indeed, our main result assumes globally quasi-uniform meshes ($\underline{h}\simeq h$).
But since some technical results hold for more general meshes we use the notation
of $\underline{h}$.
We introduce discrete spaces on sub-domains consisting of piecewise (bi)linear functions:
\[
   X_{h,j} :=
   \{ v \in C^0(\G_j);\; v|_K\ \mbox{is a polynomial of degree one } \forall K \in \CT_j\},\
   j=1,\ldots, J.
\]
Our global approximation space then is
\[
	X_h := \Pi_j X_{h,j}.
\]
We identify both product spaces $H^s(\CT)$ and $X_h$ with their direct sums,
e.g., $X_h=X_{h,1}\oplus\cdots\oplus X_{h,J}$ so as to consider their elements as scalar functions.
Doing so, we note that $X_h\not\subset \tH^{1/2}(\G)$ due to the possible discontinuity and
non-vanishing trace on $\partial\G$ of its elements.
Using the discrete space $X_h$ for the approximation of \eqref{IE} requires a different sesquilinear
form that is well defined for such functions and that controls their jumps.

For given $\nu>0$ and $r\in\R$, we define the following sesquilinear form on $X_h \times X_h$:
\begin{align*} 
   A_r(v, w)
   &:= 
   \<\bV{r} \bcurlT v, \bcurlT w\>_{\CT}
   -
   r^2 \<\bV{r} \bn\, v, \bn\, w\>_\G
   \nonumber\\
   &\qquad+
   \<T_r(v), \jump{w}\>_\gamma
   +
   \<\jump{v}, T_{-r}(w)\>_\gamma
   +
   \nu \<\jump{v}, \jump{w}\>_\gamma
\end{align*}
with operator $T_r$ being given by (cf.~\eqref{def_t})
\begin{equation} \label{def_T}
   T_r(v) := \bt(\bV{r} \bcurlT v)|_\gamma,
   \quad
   v\in H^s(\CT),\ s>1/2,\ r\in\R.
\end{equation}

The Nitsche-based non-conforming domain decomposition method associated to problem \eqref{weak_org}
then reads as: {\em Find $u_h\in X_h$ such that}
\begin{equation} \label{bem_nitsche}
   A_k(u_h,v) = \<f, v\>_\G \quad\forall v\in X_h.
\end{equation}

The analysis of this scheme will be based on a third assumption which is quite natural
but whose proof we have not found in the literature for our precise situation.
It is well known that $V_k$ and $W_k$ are Fredholm operators of index zero. This follows
from the fact that they are, respectively, compact perturbations of the positive
definite operators $V_0$ and $W_0$ as mappings of their energy spaces to the dual spaces.
For a closed smooth surface this follows from the theory of pseudo-differential operators
and has been extended by Stephan \cite{Stephan_87_BIE} to open surfaces.
There, it is shown that
\begin{align} \label{compact}
   \exists \rcomp\in (0,1/2]:\quad
   V_k-V_0:\; \tH^{-1/2}(\G)\to H^{1/2+\rcomp}(\G)
\end{align}
for a smooth open surface. Our assumption is that this holds for our open, piecewise plane
Lipschitz surface.

\medskip
\noindent{\bf Assumption 3.} \emph{There holds \eqref{compact}.}

\medskip
The main result of this paper is:

\begin{theorem} \label{thm_Cea}
Let Assumptions~1,2,3 hold true and assume that the meshes
defining $X_h$ are globally quasi-uniform, i.e., $h\simeq\underline{h}$. Given $\epsilon>0$ choose
$\nu\simeq h^{-\epsilon}$. Then the discrete scheme \eqref{bem_nitsche} is uniquely solvable
for $h$ small enough. Furthermore, selecting $\epsilon'>\epsilon$ and $s\in (1/2,\rreg]$,
there exists $h_0>0$ such that there holds the almost quasi-optimal error estimate
\[
   \|u-u_h\|_{H^{1/2}(\CT)}
   \lesssim
   h^{1/2-s-\epsilon'}
   \inf_{v\in X_h} \|u-v\|_{H^s(\CT)}
   \quad\forall h\le h_0.
\]
Here, $u$ and $u_h$ are the solutions of \eqref{IE} and \eqref{bem_nitsche}, respectively.
\end{theorem}

A proof of this result will be given at the end of Section~\ref{sec_proofs}.
We also obtain the following a priori error estimate.

\begin{corollary} \label{cor_error}
Let Assumptions~1,2,3 hold true and assume that the meshes
defining $X_h$ are globally quasi-uniform.
Given $\epsilon'>\epsilon>0$ choose $\nu\simeq h^{-\epsilon}$.
Then there exists $h_0>0$ such that there holds
\[
   \|u-u_h\|_{H^{1/2}(\CT)}
   \lesssim
   h^{\rreg-1/2-\epsilon'} \|u\|_{H^{\rreg}(\G)}
   \quad\forall h\le h_0.
\]
\end{corollary}

\begin{proof}
We combine the error estimate by Theorem~\ref{thm_Cea} with standard approximation properties.
The assertion follows by selecting $s=1/2+\epsilon''$ with $\epsilon''>0$ and renaming
$2\epsilon''+\epsilon'$ as a new $\epsilon'$.
\end{proof}

\section{Technical details and proof of the main theorem} \label{sec_proofs}

We start with collecting some preliminary technical results in the following subsection.
Then, in Subsection~\ref{sub_consistent}, we prove essential ingredients of the
proof of Theorem~\ref{thm_Cea}, which is given at the end of this section.

\subsection{Preliminary results}

We will make use of the continuity (see \cite{Costabel_88_BIO}):
\begin{equation}
\label{cont_V}
   V_r:\; \tH^{s-1}(\G) \rightarrow H^{s}(\G), \quad 0 \le s\le 1,\ r\in\R.
\end{equation}
Proofs for the statements of the following lemma can be found in
\cite[Lemma 5]{Heuer_01_ApS} and \cite[Lemma 4.3]{GaticaHH_09_BLM}.

\begin{lemma}
Let $R\subset\R^2$ be a Lipschitz domain with boundary $\partial R$.\\
(i) There holds
\begin{align} \label{tech1_1}
   \|v\|_{\tH^s(R)} &\lesssim \frac{1}{1/2-|s|} \|v\|_{H^s(R)}
\end{align}
for any $s \in (-1/2,1/2)$ and any $v \in H^s(R)$.\\
(ii) There holds
\begin{align} \label{tech1_2}
   \|v\|_{L^2(\partial R)} &\lesssim \frac{1}{\sqrt{s-1/2}} \|v\|_{H^s(R)}
\end{align}
for any $s \in (1/2,1]$ and any $v \in H^s(R)$.
\end{lemma}

\begin{lemma}
For $r\in\R$ there holds
\begin{align} \label{tech2_2}
   \|T_r v\|_{L^2(\gamma)} &\lesssim (s-1/2)^{-3/2} |v|_{H^s(\CT)}
   \quad\forall v\in H^s(\CT),\ 1/2<s\le 1,
\end{align}
with hidden constant depending on $r$.
\end{lemma}

\begin{proof}
For the case $r=0$, this estimate has been shown in \cite[Lemma 4.2]{ChoulyH_12_NDD}.
Using the continuity \eqref{cont_V} for wave number $r\not=0$, the same estimates apply.
\end{proof}

\begin{lemma} There holds
\begin{align} \label{tech3_1}
   \|\bcurlT v\|_{\tH^{-1/2}(\CT)} \lesssim (s-1/2)^{-1} |v|_{H^s(\CT)}
   \quad\forall v\in H^{s}(\CT),\ 1/2< s \le 1,
\end{align}
\begin{align} \label{tech3_2}
   \|\bcurlT v\|_{H^{-1/2}(\CT)} \gtrsim |v|_{H^{1/2}(\CT)}
   \quad\forall v\in H^{1/2}(\CT).
\end{align}
\end{lemma}

\begin{proof}
The first estimate can be proved by using the equivalence \eqref{tech1_1}
of $\tH^{s}(\G_j)$ and $H^s(\G_j)$-norms
for $s\in (-1/2,0]$, the continuity of $\bcurlS{\G_j}:\;H^s(\G_j)\to H^{s-1}(\G_j)$
($s\in (1/2,1]$) and a quotient-space argument (cf.~\cite{Heuer_14_OEF}).
For details see \cite[(4.15)]{ChoulyH_12_NDD}.
Bound \eqref{tech3_2} follows by applying face-wise the corresponding estimate from
\cite[Lemma 4.1]{GaticaHH_09_BLM}.
\end{proof}

\begin{lemma} \label{la_tech4}
For $\nu>0$ and $s>1/2$ there holds
\begin{align} \label{tech4}
   \|v\|_{L^2(\G)} \lesssim \|v\|_{H^s_\nu(\CT)}
   \quad\forall v\in H^s(\CT).
\end{align}
\end{lemma}

\begin{proof}
The proof is a slight variation of the proof of the Poincar\'e-Friedrichs
inequality by a compactness argument. In this case, we use the compactness of
the embedding of $H^s(\CT)$ in $L^2(\G)$ and the boundedness of the functional
$\<\jump{\cdot},\jump{\cdot}\>_\gamma^{1/2}$ on $H^s(\CT)$. Furthermore, the kernel
of $|\cdot|_{H^s(\CT)}$ consists of $\CT$-piecewise constant functions which are
eliminated by the functional $\<\jump{\cdot},\jump{\cdot}\>_\gamma^{1/2}$
(note that the jump $\jump{\cdot}$ reduces to the trace operator on $\partial\G$).
\end{proof}

\subsection{Consistency, boundedness, discrete ellipticity, and Aubin-Nitsche trick}
\label{sub_consistent}

In this section we show four essential ingredients for the proof of Theorem~\ref{thm_Cea}.
These are the consistency of the non-conforming discrete scheme (Lemma~\ref{la_consistent}),
the boundedness of the sesquilinear form (Lemma~\ref{la_cont}), its discrete ellipticity
in the form of G\r{a}rding inequalities with respect to two norms
(Lemma~\ref{la_ell} and Corollary~\ref{cor_ell}), and an error estimate in a lower-order norm
based on the Aubin-Nitsche trick (Lemma~\ref{la_AN}).

\begin{lemma} \label{la_consistent}
Let Assumption~1 hold true. Then, for $\nu>0$, the discrete scheme \eqref{bem_nitsche} is consistent.
That is, the solution $u$ of \eqref{IE} satisfies
\[
   A_k(u,v) = \<f, v\>_\G \quad\forall v\in X_h.
\]
\end{lemma}

\begin{proof}
By Assumption~1, $u\in\tH^r(\G)$ for an $r>1/2$. In particular, $u$ is continuous and vanishes on
$\partial\G$ in the sense of traces. It follows that
\[
   A_k(u,v) =
   \<\bV{k} \bcurlT u, \bcurlT v\>_{\CT}
   -
   k^2 \<\bV{k} \bn\, u, \bn\, v\>_\G
   +
   \<T_k(u), \jump{v}\>_\gamma
   \quad\forall v\in X_h.
\]
The integration-by-parts formula \eqref{def_t} holds for $\bvarphi:=\bV{k}\bcurlT u$
and $v\in X_h$ (see \cite{GaticaHH_09_BLM,ChoulyH_12_NDD} for details concerning
the Laplacian; they also apply to the Helmholtz case).
The definition \eqref{def_T} of $T_k$ and relation \eqref{WV} then show the assertion.
\end{proof}

\begin{lemma} \label{la_cont}
There holds
\begin{align*}
   |A_k(v,w)|
   &\lesssim 
   \max\{\nu (s-1/2)^{-1}, (s-1/2)^{-2}\} \|v\|_{H^{s}(\CT)} \|w\|_{H^{s}(\CT)}
   \quad\forall v, w \in H^s(\CT),\ s>1/2.
\end{align*}
\end{lemma}

\begin{proof}
By the continuity of $\bV{k}$, and \eqref{tech3_1} we obtain
\begin{align*}
   |\<\bV{k}\bcurlT v, \bcurlT w\>_\G|
   &\lesssim
   \|\bcurlT v\|_{\tH^{-1/2}(\G)} \|\bcurlT w\|_{\tH^{-1/2}(\G)}\\
   &\lesssim
   (s-1/2)^{-2} |v|_{H^s(\CT)} |w|_{H^s(\CT)}
   \quad\forall v, w\in H^s(\CT),\ s>1/2,
\end{align*}
and
\begin{align} \label{pf_cont_1}
   |\<\bV{k} \bn\,v, \bn\,w\>|
   &\lesssim
   \|\bn\,v\|_{\tH^{-1/2}(\G)} \|\bn\,w\|_{\tH^{-1/2}(\G)}
   \lesssim
   \|v\|_{L^2(\G)} \|w\|_{L^2(\G)}
   \quad\forall v, w\in L^2(\G).
\end{align}
Combinations of \eqref{tech1_2} with the Cauchy-Schwarz and triangle inequalities, and estimate
\eqref{tech2_2}, yield
\begin{align*}
   |\<\jump{v},\jump{w}\>_\gamma|
   &\lesssim
   (s-1/2)^{-1} \|v\|_{H^s(\CT)} \|w\|_{H^s(\CT)}
   \quad\forall v, w\in H^s(\CT),\ s>1/2
\end{align*}
and
\begin{align*}
   |\<T_k(v), \jump{w}\>_\gamma|
   +
   |\<\jump{v}, T_{-k}(w)\>_\gamma|
   &\lesssim
   (s-1/2)^{-2} \|v\|_{H^s(\CT)} \|w\|_{H^s(\CT)}
   \quad\forall v, w\in H^s(\CT),\ s>1/2.
\end{align*}
The previous bounds prove the assertion.
\end{proof}

\begin{lemma} \label{la_ell}
Let Assumption~3 hold true.
For any $\epsilon>0$ there exists $c_G>0$ such that for
$\nu\gtrsim\underline{h}^{-\epsilon}$ there holds
\[
   |A_k(v,v)|
   \gtrsim
   \|v\|_{H^{1/2}_\nu(\CT)}^2
   -
   c_G (s-1/2)^{-2} \underline{h}^{1-2s}
   \|v\|_{H^{1/2-\rcomp}(\CT)}^2
   \quad\forall v\in X_h,\ s\in (1/2,1].
\]
\end{lemma}

\begin{proof}
Application of \eqref{tech3_2}, \eqref{pf_cont_1},
and the fact that $\nu\gtrsim 1$ prove that there exists $c_G>0$
($G$ refers to G\r{a}rding) such that
\begin{align} \label{pf_ell_1}
   |\<\bV{0} \bcurlT v, \bcurlT v\>_{\CT}
   - &k^2 \<\bV{k} \bn\,v, \bn\,v\>_\G
   + \nu \<\jump{v}, \jump{v}\>_\gamma|
   \nonumber\\
   &\gtrsim
   \|v\|_{H^{1/2}_\nu(\CT)}^2 - c_G\|v\|_{L^2(\G)}^2\quad\forall v\in X_h.
\end{align}
We are left with bounding the remaining terms.

By Assumption~3, \eqref{tech3_1} and the inverse property we can bound
\begin{align} \label{pf_ell_2}
   |\<(\bV{k}-\bV{0}) \bcurlT v, \bcurlT v\>_{\CT}|
   &\lesssim
   \|\bcurlT v\|_{\tH^{-1/2}(\G)}
   \|\bcurlT v\|_{\tH^{-1/2-\rcomp}(\G)}
   \nonumber\\
   &\lesssim
   (s-1/2)^{-1}
   |v|_{H^s(\CT)}
   \|\bcurlT v\|_{\tH^{-1/2-\rcomp}(\G)}
   \nonumber\\
   &\lesssim
   (s-1/2)^{-1}
   \underline{h}^{1/2-s}
   |v|_{H^{1/2}(\CT)}
   \|\bcurlT v\|_{\tH^{-1/2-\rcomp}(\G)},
\end{align}
for any $v\in X_h$ and $s\in(1/2,1]$. In order to estimate the last term above we use that
there holds
\[
   \|\bcurl_j v_j\|_{\tH^{-1/2-\rcomp}(\G_j)}
   \lesssim
   \|v_j\|_{\tH^{1/2-\rcomp}(\G_j)},\ j=1,\ldots, J.
\]
This follows from Fourier analysis, considering $\G_j$ as a sub-domain of $\R^2$,
and since $\|\phi\|_{\tH^t(\G_j)}\simeq \|\phi^0\|_{H^t(\R^2)}$ for $t\in[-1,1]$
and $\phi\in C_0^\infty(\G_j)$ with $\phi^0$ denoting ist extension by $0$.
By a standard domain decomposition estimate and bound \eqref{tech1_1} we then conclude that
\begin{align} \label{pf_ell_3}
   \|\bcurlT v\|_{\tH^{-1/2-\rcomp}(\G)}
   &\lesssim
   \|\bcurlT v\|_{\tH^{-1/2-\rcomp}(\CT)}
   \lesssim
   \|v\|_{\tH^{1/2-\rcomp}(\CT)}
   \lesssim
   \rcompinv
   \|v\|_{H^{1/2-\rcomp}(\CT)}
\end{align}
for any $v\in X_h$.
Combination of \eqref{pf_ell_2} and \eqref{pf_ell_3}, and Young's inequality, prove that
\begin{align} \label{pf_ell_4}
   |\<(\bV{k}-\bV{0}) \bcurlT v, \bcurlT v\>_{\CT}|
   &\lesssim
   \rho
   |v|_{H^{1/2}(\CT)}^2
   +
   \rho^{-1} \rcompinvt
   (s-1/2)^{-2} \underline{h}^{1-2s}
   \|v\|_{H^{1/2-\rcomp}(\CT)}^2
\end{align}
for any $v\in X_h$, $s\in (1/2,1]$, and $\rho>0$.
The two remaining terms are included analogously as in the case of the Laplacian
($k=0$) considered in \cite[Lemma~4.4]{ChoulyH_12_NDD}. Specifically, using
\eqref{tech2_2} with $s=1/2+\epsilon$, Young's inequality and the inverse property,
one proves that
\begin{align} \label{pf_ell_20}
   |\<T_k v, \jump{v}\>_\gamma| + |\<\jump{v}, T_{-k} v\>_\gamma|
   &\lesssim
   \underline{h}^{-2\epsilon} \frac\delta{\epsilon^3} |v|_{H^{1/2}(\CT)}^2
   +
   \frac 1\delta \|\jump{v}\|_{L^2(\gamma)}^2
   \quad\forall v\in X_h,\ \delta>0,\ \epsilon>0.
\end{align}
A combination of \eqref{pf_ell_1}, \eqref{pf_ell_4} with $\rho$ small enough,
and \eqref{pf_ell_20} shows that there exist constants $c_1, c_2, c_G>0$
($c_G$ possibly different from before) such that
\begin{align*}
   |A_k(v,v)|
   &\gtrsim
   \Bigl(1 - c_1\underline{h}^{-2\epsilon} \frac\delta{\epsilon^3}\Bigr)
   |v|_{H^{1/2}(\CT)}^2
   +
   \Bigl(\nu - \frac {c_2}{\delta}\Bigr)
   \|\jump{v}\|_{L^2(\gamma)}^2
   -
   c_G (s-1/2)^{-2} \underline{h}^{1-2s}
   \|v\|_{H^{1/2-\rcomp}(\CT)}^2
\end{align*}
for any $v\in X_h$, $\delta>0, \epsilon>0$, $s\in (1/2,1]$.
With $\delta:=\underline{h}^{3\epsilon}$, selecting $\nu\gtrsim \underline{h}^{-4\epsilon}$,
and replacing $4\epsilon$ by $\epsilon$ we obtain the assertion.
\end{proof}

\begin{corollary} \label{cor_ell}
Let Assumption~3 hold true, and let $\epsilon,\epsilon'>0$ be given.
There exists $c_G>0$ such that for
$\nu\gtrsim\underline{h}^{-\epsilon}$ there holds
\[
   |A_k(v,v)|
   \gtrsim
   \underline{h}^{2\epsilon'} \|v\|_{H^{1/2}(\CT)}^2
   -
   c_G (s-1/2)^{-2} \underline{h}^{1-2s}
   \|v\|_{H^{1/2-\rcomp}(\CT)}^2
   \quad\forall v\in X_h,\ s\in (1/2,1].
\]
\end{corollary}

\begin{proof}
We select $s=1/2+\epsilon'$ in \eqref{tech4} and use the inverse property to conclude that
\[
   \|v\|_{L^2(\G)}^2
   \lesssim
   \underline{h}^{-2\epsilon'} |v|_{H^{1/2}(\G)}^2 + \nu \|\jump{v}\|_{L^2(\gamma)}^2
   \quad\forall v\in X_h.
\]
This means that
$\underline{h}^{2\epsilon'} \|v\|_{H^{1/2}(\CT)}^2\lesssim \|v\|_{H^{1/2}_\nu(\CT)}^2$
for any $v\in X_h$, and the assertion follows from Lemma~\ref{la_ell}.
\end{proof}

\begin{lemma} \label{la_AN}
Let Assumptions~1,2, and~3 hold true and assume that the meshes
defining $X_h$ are globally quasi-uniform, i.e., $h\simeq\underline{h}$. Given $\epsilon>0$ choose
$\nu\simeq h^{-\epsilon}$. Then there exists $h_0>0$ such that the discrete scheme
\eqref{bem_nitsche} is uniquely solvable for $h\le h_0$. Furthermore, for $h\le h_0$ there holds
\begin{align*}
   \|u-u_h\|_{H^{1-\riso}(\CT)}
   \lesssim
   h^{\riso-s} \max\{\nu (s-1/2)^{-1}, (s-1/2)^{-2}\} \|u-u_h\|_{H^{s}(\CT)}
\end{align*}
for any $s\in (1/2,\min\{\riso,\rreg\}]$.
\end{lemma}

\begin{proof}
We first show the error estimate, i.e., for the time being let us assume that there is a
(unique) solution $u_h$ to \eqref{bem_nitsche}.
Note that there holds $A_r(v,w)=\overline{A_{-r}(w,v)}$ for any $r\in\R$ and
sufficiently smooth functions $v$, $w$. This follows from the fact the $\bV{-r}$ is the adjoint
operator of $\bV{r}$.
Let $\phi\in\tH^\riso(\G)$ be given (cf. Assumption~2).
By standard approximation results there exists $\phi_h\in X_h$ such that
\begin{equation} \label{pf_AN_1}
   \|\phi-\phi_h\|_{H^{s}(\CT)} \lesssim h^{\riso-s} \|\phi\|_{\tH^\riso(\G)},
   \quad 1/2\le s\le \riso.
\end{equation}
Using integration by parts (analogously to proving consistency in Lemma~\ref{la_consistent})
we find that there holds
\begin{align*}
   \<u-u_h,W_{-k}\phi\>_\G
   &=
   \overline{\<W_{-k}\phi, u-u_h\>}
   =
   \overline{A_{-k}(\phi, u-u_h)}
   =
   A_k(u-u_h,\phi) = A_k(u-u_h,\phi-\phi_h).
\end{align*}
Lemma~\ref{la_cont} and \eqref{pf_AN_1} then prove that
\begin{align*}
   |\<u-u_h,W_{-k}\phi\>_\G|
   &\lesssim
   \max\{\nu (s-1/2)^{-1}, (s-1/2)^{-2}\}
   \|u-u_h\|_{H^{s}(\CT)} h^{\riso-s} \|\phi\|_{\tH^\riso(\G)}.
\end{align*}
Noting that $\|\cdot\|_{\tH^{1-\riso}(\G)}\simeq \|\cdot\|_{H^{1-\riso}(\CT)}$
by \eqref{tech1_1} since $\riso\in (1/2,1)$, this bound implies the error estimate
via duality and by making use of Assumption~2:
\begin{align*}
   \|u-u_h\|_{\tH^{1-\riso}(\G)}
   \simeq
   \sup_{0\not=\psi\in H^{\riso-1}(\G)}
   \frac {|\<u-u_h,\psi\>_\G|}
         {\|\psi\|_{H^{\riso-1}(\G)}}
   &\simeq
   \sup_{0\not=\phi\in\tH^\riso(\G)}
   \frac {|\<u-u_h,W_{-k}\phi\>_\G|}
         {\|\phi\|_{\tH^\riso(\G)}}.
\end{align*}
We are left with showing unique existence of $u_h$ for small $h$. Since we are dealing
with a quadratic discrete system, it is enough to show uniqueness. Therefore, in the
remainder of this proof, we assume that we are solving the homogeneous problem \eqref{weak_org},
i.e., $f=0$ and $u=0$. We have to show that only $u^0_h=0$ solves the homogeneous discrete scheme
\eqref{bem_nitsche}. The first part of this proof and the inverse property show that there holds
\begin{align*}
   \|u^0_h\|_{H^{1-\riso}(\CT)}
   &\lesssim
   h^{\riso-s} \max\{\nu (s-1/2)^{-1}, (s-1/2)^{-2}\} \|u^0_h\|_{H^{s}(\CT)}
   \\
   &\lesssim
   h^{\riso+1/2-2s} \max\{\nu (s-1/2)^{-1}, (s-1/2)^{-2}\} \|u^0_h\|_{H^{1/2}(\CT)}
\end{align*}
for $s\in (1/2,1]$.
Now, if $\riso\le 1/2+\rcomp$, then
\(
   \|u^0_h\|_{H^{1/2-\rcomp}(\CT)} \le \|u^0_h\|_{H^{1-\riso}(\CT)},
\)
and combination of the bound above with the estimate by Corollary~\ref{cor_ell} yields
\begin{equation} \label{pf_AN_2}
   0=|A_k(u^0_h,u^0_h)|
   \gtrsim
   \Bigl(h^{2\epsilon'} - C(s) h^{2\riso+2-6s-2\epsilon}\Bigr)
   \|u^0_h\|_{H^{1/2}(\CT)}^2
\end{equation}
for a number $C(s)$ depending on $s$.
For $\epsilon, \epsilon'>0$ small enough,
we can select $s\in \bigl(1/2, (\riso+1-\epsilon-\epsilon')/3\bigr)$ and find $h_0>0$
such that
\[
   h^{2\epsilon'} - C(s) h^{2\riso+2-6s-2\epsilon}
   =
   h^{2\epsilon'} \Bigl(1 - C(s) h^{2\riso+2-6s-2\epsilon-2\epsilon'}\Bigr)
   > 0\quad\forall h\le h_0.
\]
If $\riso>1/2+\rcomp$ then we additionally use the inverse property to bound
\(
   \|u^0_h\|_{H^{1/2-\rcomp}(\CT)} \le h^{1/2+\rcomp-\riso} \|u^0_h\|_{H^{1-\riso}(\CT)}.
\)
Then we obtain, instead of \eqref{pf_AN_2},
\begin{equation} \label{pf_AN_3}
   0=|A_k(u^0_h,u^0_h)|
   \gtrsim
   \Bigl(h^{2\epsilon'} - C(s) h^{2\rcomp+3-6s-2\epsilon}\Bigr)
   \|u^0_h\|_{H^{1/2}(\CT)}^2.
\end{equation}
Analogously as before, for $\epsilon, \epsilon'>0$ small enough,
we can select $s\in \bigl(1/2, 1/2+(\rcomp-\epsilon-\epsilon')/3\bigr)$ and
find $h_0>0$ such that
\[
   h^{2\epsilon'} - C(s) h^{2\rcomp+3-6s-2\epsilon}
   =
   h^{2\epsilon'} \Bigl(1 - C(s) h^{2\rcomp+3-6s-2\epsilon-2\epsilon'}\Bigr)
   > 0\quad\forall h\le h_0.
\]
In both cases, \eqref{pf_AN_2} respectively \eqref{pf_AN_3} proves that
$u^0_h=0$ for $h$ sufficiently small.
\end{proof}

\subsection{Proof of Theorem~\ref{thm_Cea}}

The existence and uniqueness of $u_h$ solving \eqref{bem_nitsche} for small $h$ is guaranteed
by Lemma~\ref{la_AN}.
The proof of the error estimate follows the standard Strang strategy, that is, adding and subtracting
a discrete function, using the triangle inequality, discrete G\r{a}rding's inequality
(Corollary~\ref{cor_ell}), consistency (Lemma~\ref{la_consistent}) and boundedness
(Lemma~\ref{la_cont}).
More precisely, given $\epsilon'>\epsilon>0$, $h$ small enough, $\nu\simeq h^{-\epsilon}$,
and $v\in X_h$, we find by the just mentioned arguments
\begin{align} \label{pf_Cea_0}
\lefteqn{
   \|u-u_h\|_{H^{1/2}(\CT)}^2
   \lesssim
   \|u-v\|_{H^{1/2}(\CT)}^2 + \|u_h-v\|_{H^{1/2}(\CT)}^2
}
   \nonumber\\
   &\lesssim
   \|u-v\|_{H^{1/2}(\CT)}^2
   +
   (s-1/2)^{-2} h^{1-2s-2\epsilon'} \|u_h-v\|_{H^{1/2-\rcomp}(\CT)}^2
   +
   h^{-2\epsilon'} |A_k(u-v,u_h-v)|
   \nonumber\\
   &\lesssim
   \|u-v\|_{H^{1/2}(\CT)}^2
   +
   (s-1/2)^{-2} h^{1-2s-2\epsilon'} \|u_h-v\|_{H^{1/2-\rcomp}(\CT)}^2
   \nonumber\\
   &\ +
   h^{-2\epsilon'} \max\{\nu (s-1/2)^{-1}, (s-1/2)^{-2}\}
   \Bigl(\delta^{-1} \|u-v\|_{H^{s}(\CT)}^2
         +
         \delta h^{1-2s} \|u_h-v\|_{H^{1/2}(\CT)}^2
   \Bigr)
\end{align}
for any $\delta>0$. In the last step we applied Young's and the inverse inequality.
We have to consider the relation between $\riso$ and $\rcomp$,
cf. the proof of Lemma~\ref{la_AN}.
In the case $\riso\le 1/2+\rcomp$ we bound
\(
   \|u_h-v\|_{H^{1/2-\rcomp}(\CT)} \le \|u_h-v\|_{H^{1-\riso}(\CT)}.
\)
Otherwise,
\[
   \|u_h-v\|_{H^{1/2-\rcomp}(\CT)} \lesssim h^{1/2+\rcomp-\riso} \|u_h-v\|_{H^{1-\riso}(\CT)}
\]
by the inverse property. Both cases are considered by
\[
   \|u_h-v\|_{H^{1/2-\rcomp}(\CT)}
   \lesssim
   h^{-\alpha} \|u_h-v\|_{H^{1-\riso}(\CT)},
   \quad\alpha:=\max\{0,\riso-\rcomp-1/2\}.
\]
We use the just established estimate in \eqref{pf_Cea_0},
and bound the norms of $u_h-v$ by adding and subtracting $u$,
applying the triangle inequality,
and then bound $\|u-u_h\|_{H^{1-\riso}(\CT)}$ with the help of Lemma~\ref{la_AN}.
This yields (with hidden constants depending on $s$)
\begin{align} \label{pf_Cea_1}
\lefteqn{
   \|u-u_h\|_{H^{1/2}(\CT)}^2
   \lesssim
   \|u-v\|_{H^{1/2}(\CT)}^2
   +
   h^{1-2s-2\epsilon'-2\alpha} \|u-v\|_{H^{1-\riso}(\CT)}^2
}\nonumber\\
   &+
   h^{-2\epsilon'} \nu
   \Bigl(\delta^{-1} \|u-v\|_{H^{s}(\CT)}^2
         +
         \delta h^{1-2s} \|u-v\|_{H^{1/2}(\CT)}^2
   \Bigr)\nonumber\\
   &+
   \delta h^{1-2s-2\epsilon'} \nu \|u-u_h\|_{H^{1/2}(\CT)}^2
   +
   h^{1+2\riso-4s-2\epsilon'-2\alpha} \nu^2 \|u-u_h\|_{H^{s}(\CT)}^2.
\end{align}
The last term is handled yet again by the same technique (adding and subtracting $v$,
inverse property of $u_h-v$):
\begin{align} \label{pf_Cea_2}
   \|u-u_h\|_{H^{s}(\CT)}^2
   &\lesssim
   \|u-v\|_{H^{s}(\CT)}^2
   +
   h^{1-2s} \Bigl(\|u-v\|_{H^{1/2}(\CT)}^2 + \|u-u_h\|_{H^{1/2}(\CT)}^2\Bigr).
\end{align}
Combination of \eqref{pf_Cea_1} and \eqref{pf_Cea_2}, and reordering terms yields
\begin{align*}
   \lefteqn{
   \Bigl(1 - c_1\delta h^{1-2s-2\epsilon'} \nu
           - c_2h^{2+2\riso-6s-2\epsilon'-2\alpha} \nu^2\Bigr)
   \|u-u_h\|_{H^{1/2}(\CT)}^2
   }\\
   &\lesssim
   h^{1-2s-2\epsilon'-2\alpha} \|u-v\|_{H^{1-\riso}(\CT)}^2\\
   &\ +
   \Bigl(
      1 + h^{1-2s-2\epsilon'} \nu \delta
        + h^{2+2\riso-6s-2\epsilon'-2\alpha} \nu^2
   \Bigr)
   \|u-v\|_{H^{1/2}(\CT)}^2\\
   &\ +
   \Bigl(  h^{-2\epsilon'} \nu \delta^{-1}
         + h^{1+2\riso-4s-2\epsilon'-2\alpha} \nu^2
   \Bigr)
   \|u-v\|_{H^{s}(\CT)}^2
\end{align*}
for two constants $c_1, c_2>0$. We now select
$\epsilon'>\epsilon>0$ sufficiently small such that there is
$s\in \bigl(1/2, (1+\riso-\alpha-\epsilon-\epsilon')/3\bigr)$. This is possible since
$1+\riso-\alpha>3/2$. Furthermore, for the selected $s$ we choose
$\delta=h^{2s-1+\epsilon+2\epsilon'+\epsilon''}$ for $\epsilon''>0$.
Then the factor on the left-hand side is bounded from below by a positive
constant for $h$ being small enough.
Replacing $\delta$ and $\nu$ also on the right-hand side shows that
\begin{align*}
   \|u-u_h\|_{H^{1/2}(\CT)}^2
   &\lesssim
   h^{1-2s-2\epsilon'-2\alpha} \|u-v\|_{H^{1-\riso}(\CT)}^2
   +
   \|u-v\|_{H^{1/2}(\CT)}^2\\
   &\ +
   \Bigl(  h^{1-2s-(2\epsilon+4\epsilon'+\epsilon'')}
         + h^{1+2\riso-4s-2\alpha-2\epsilon'-2\epsilon)}
   \Bigr)
   \|u-v\|_{H^{s}(\CT)}^2
\end{align*}
for $h$ small enough.
By the assumptions $\riso\in (1/2,1)$ and $\rcomp\in(0,1/2]$ one finds that,
for $\epsilon, \epsilon'$ sufficiently small, the term
\(
   h^{1-2s-(2\epsilon+4\epsilon'+\epsilon'')}
   \|u-v\|_{H^{s}(\CT)}^2
\)
is the dominating one of the upper bound in the sense of best approximation orders in $h$.
Then, renaming $\epsilon+2\epsilon'+\epsilon''/2$ to be the new $\epsilon'$,
this yields the error estimate of Theorem~\ref{thm_Cea}, for the previously noted
selection of $s\in \bigl(1/2,(1+\riso-\alpha-\epsilon-\epsilon')/3\bigr)$.
It is also clear that the upper bound for $s$ can be dropped, as long as $s\le\rreg$.

\section{Numerical results} \label{sec_num}

We consider the model problem (\ref{IE}) with $\G=(-1/2,1/2) \times (-1/2,1/2) \times \{0\}$,
right-hand side function $f=1$, and wave number $k=5$.
We use a decomposition of $\G$ into three sub-domains,
as indicated in Fig.~\ref{fig_meshes}, and consider
rectangular meshes which are piecewise uniform with respect to
sub-domains, and globally quasi-uniform.
The initial four meshes are also shown in Fig.~\ref{fig_meshes}.
The discrete spaces $X_h$ consist of piecewise bilinear polynomials which are continuous
on sub-domains.

\begin{figure}[htb]
\centering
\includegraphics[width=0.6\textwidth]{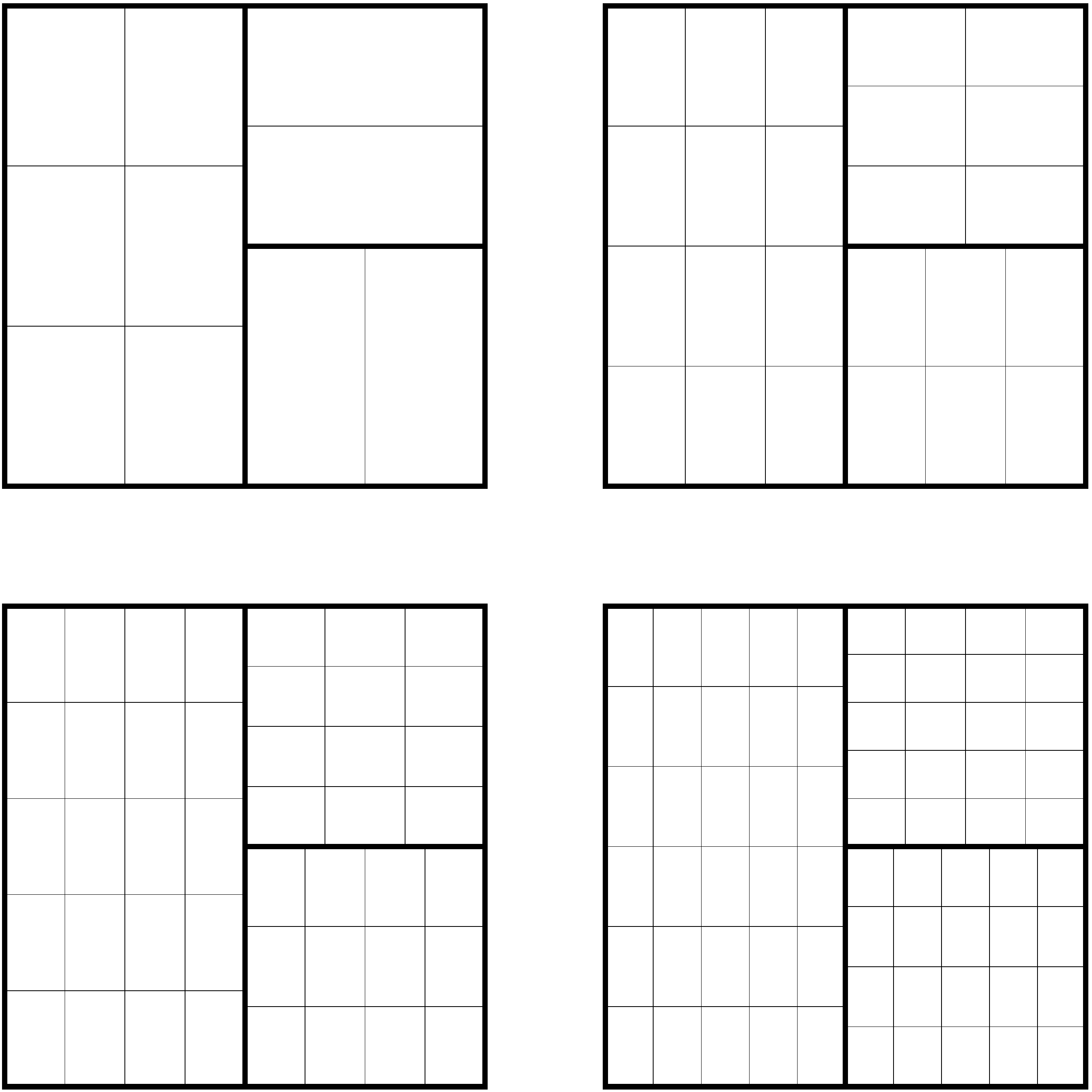} 
\caption{Decomposition of $\G$ into three sub-domains and initial mesh sequence.}
\label{fig_meshes}
\end{figure}

According to Corollary~\ref{cor_error}, and taking into account Remark~\ref{rem2}~(i),
we expect that for sufficiently large $\nu$,
the error $\|u-u_h\|_{H^{1/2}(\CT)}$ has convergence order close to
$1/2$, the optimal one for a conforming method and piecewise (bi)linear functions on
quasi-uniform meshes, cf.~\cite{BespalovH_08_hpB}.
However, since the exact solution $u$ of \eqref{IE} is unknown, the error
cannot be computed directly. But even knowing $u$ it would be difficult to calculate
the necessary norm. For the Laplacian ($k=0$) the residual and the
$L^2$-norm of the jumps form a reasonable upper bound for the error, see the discussion
in \cite[Section~5]{HeuerM_13_DGB}. In the Helmholtz case ($k\not=0$) energy arguments
leading to such estimates do not apply without perturbation terms, see
\cite[Section~5]{HolmMS_96_hpB}. Nevertheless, we conclude from the previously
mentioned discussions that the sum of the two terms
\[
   \underbrace{\Bigl|\|u\|_{\mathrm{ex}}^2
                -
                \mathrm{Re}\bigl(\<\bV{k} \bcurlT u_h, \bcurlT u_h\>_\CT
                -
                k^2\<V_k u_h, u_h\>_\G\bigr)\Bigr|^{1/2}}_{=: \mathrm{residual}}
   +
   \underbrace{\|\jump{u_h}\|_{L^2(\gamma)}}_{=: \mathrm{jumps}}
\]
is a reasonably justified upper bound for the error $\|u-u_h\|_{H^{1/2}(\CT)}$.
Here, $\|u\|_{\mathrm{ex}}^2$ is an approximation of $\mathrm{Re}\<W_k u,u\>_\G$
generated by extrapolation on a sequence of uniform meshes, cf.~\cite{ErvinHS_93_hpB}.

Figure~\ref{fig_error} shows the errors on a double logarithmic scale
versus the inverse of the maximum over all side lengths.
For comparison also the error in energy norm for the conforming variant on a sequence
of uniform meshes and the curve $0.25 h^{1/2}$ are given. They confirm the convergence
order $O(h^{1/2})$ of the conforming BEM. The results of the Nitsche approximation
with $\nu=10,100,1000$ indicate that, for $\nu$ sufficiently large ($\nu=1000$ in this
case) this optimal order is achieved. At least for the model problem, this wave number
and for the meshes considered, we do not observe a reduced convergence order.
Such a reduced order can be seen in the case $\nu=10$. For $\nu=100$ the residual
appears to reflect some pre-asymptotic behavior whereas the jumps still indicate a reduced
convergence order.

For illustration, we also present some conforming and Nitsche approximations to
the solution $u$ of \eqref{IE}, again with wave number $k=5$.
Figure~\ref{fig_approx_conf} shows a conforming approximation (including homogeneous boundary
condition) whereas Figure~\ref{fig_approx_Nitsche} presents the
Nitsche results for different meshes (the real parts on the left and the imaginary parts
on the right). The coarser mesh (upper plots) is the last one from Figure~\ref{fig_meshes}.

\begin{figure}[htb]
\includegraphics[width=\textwidth]{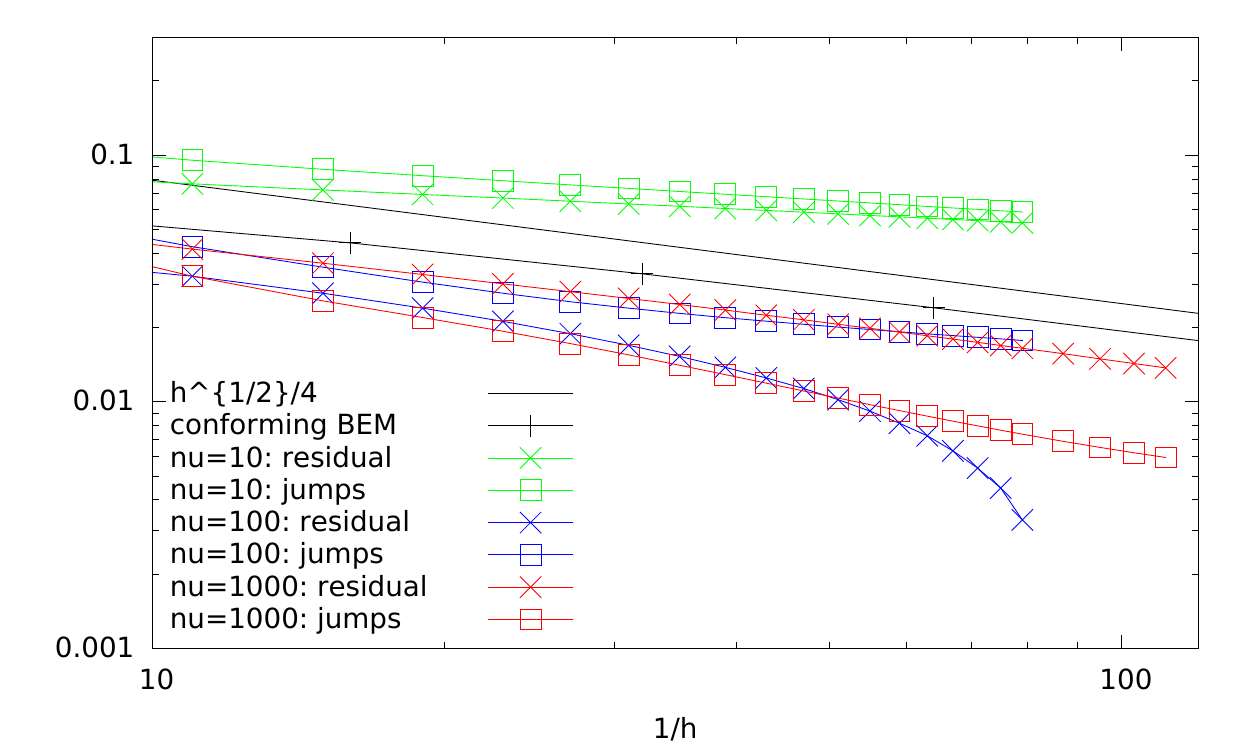}
\caption{Errors of conforming BEM and Nitsche-BEM with $\nu=10, 100, 1000$.}
\label{fig_error}
\end{figure}

\begin{figure}[htb]
\begin{center}
\includegraphics[width=0.45\textwidth]{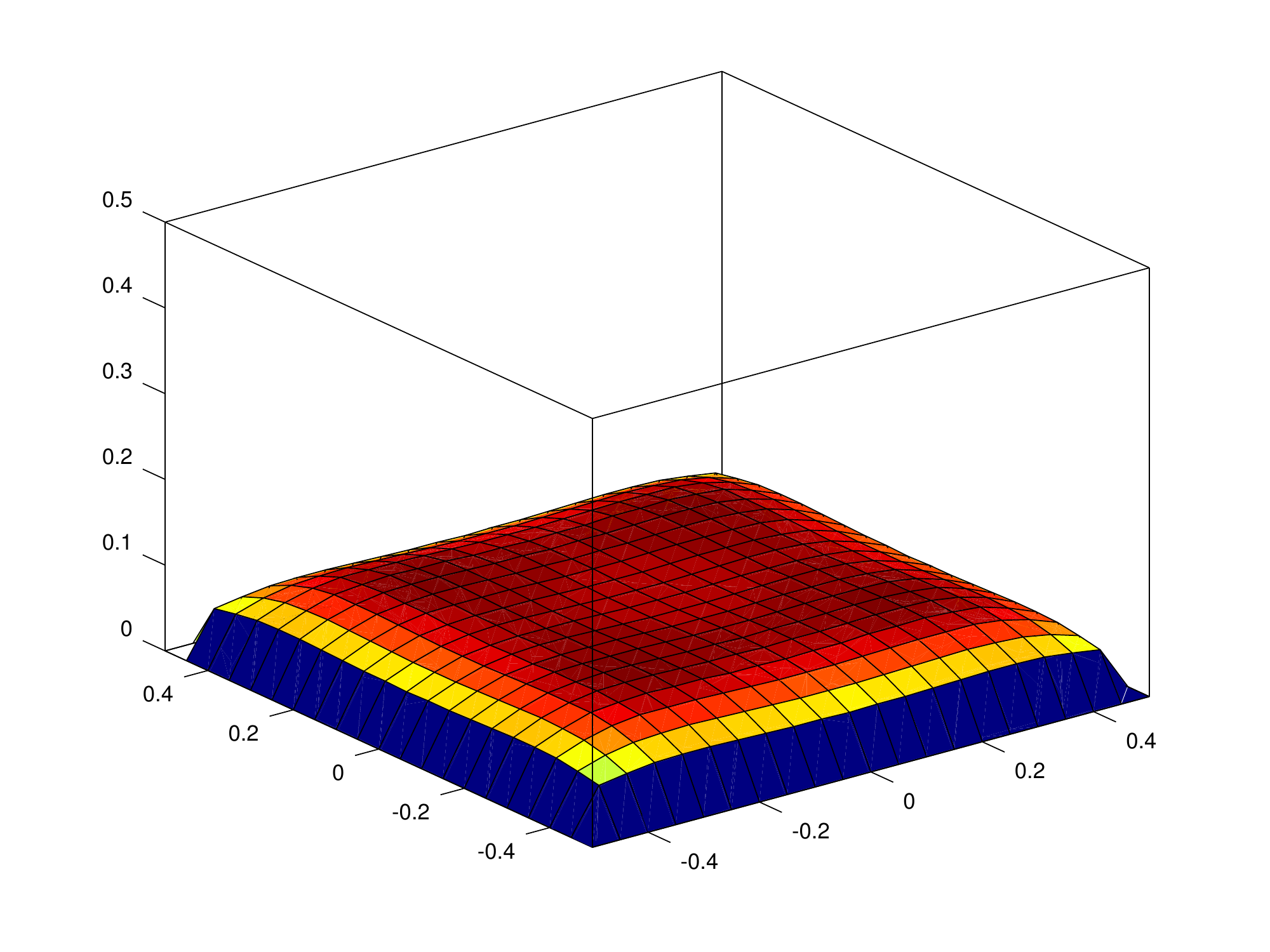}
\includegraphics[width=0.45\textwidth]{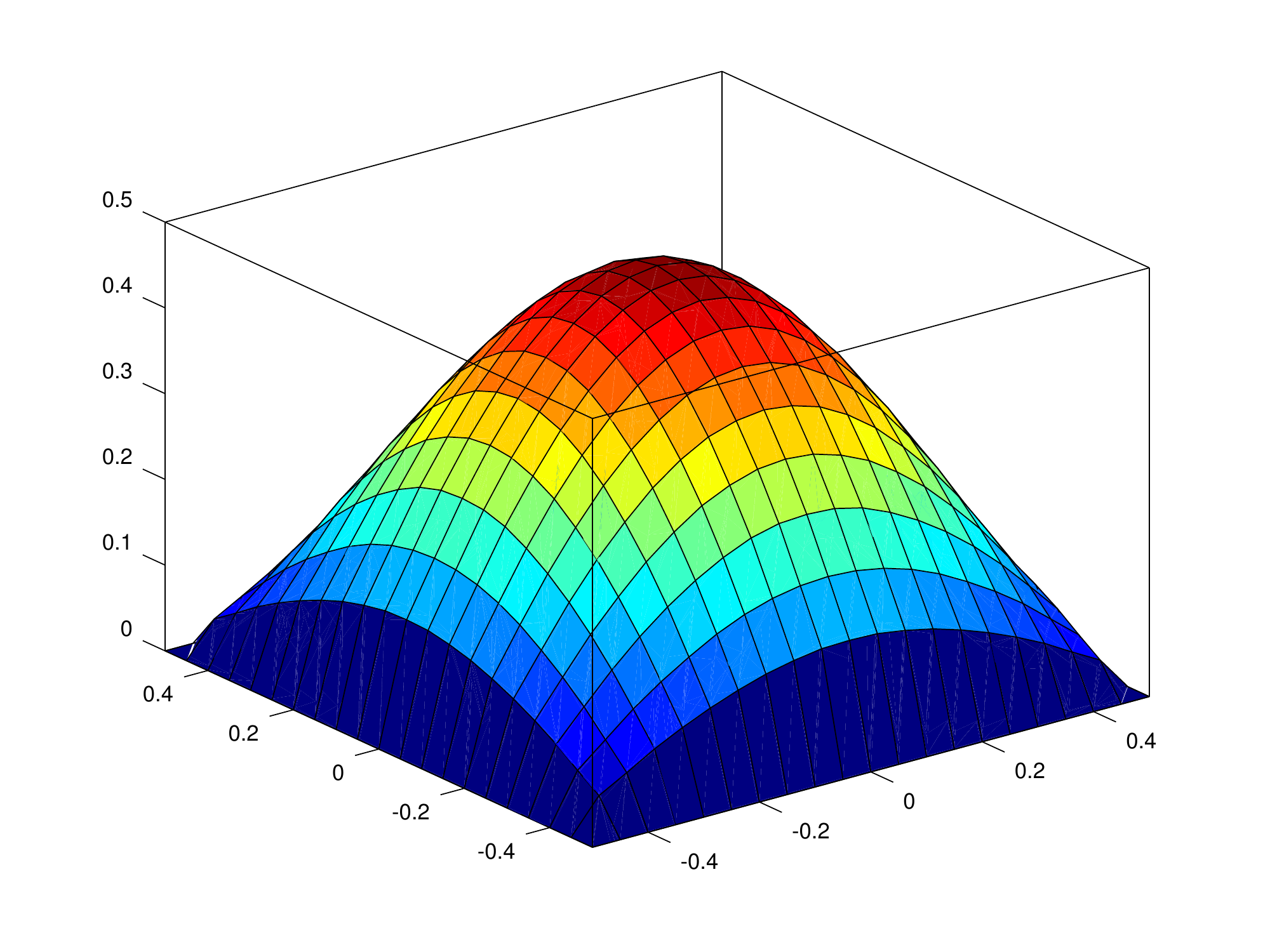}
\end{center}
\caption{Conforming approximation for $k=5$, real part (left) and imaginary part (right).}
\label{fig_approx_conf}
\end{figure}

\begin{figure}[htb]
\begin{center}
\includegraphics[width=0.45\textwidth]{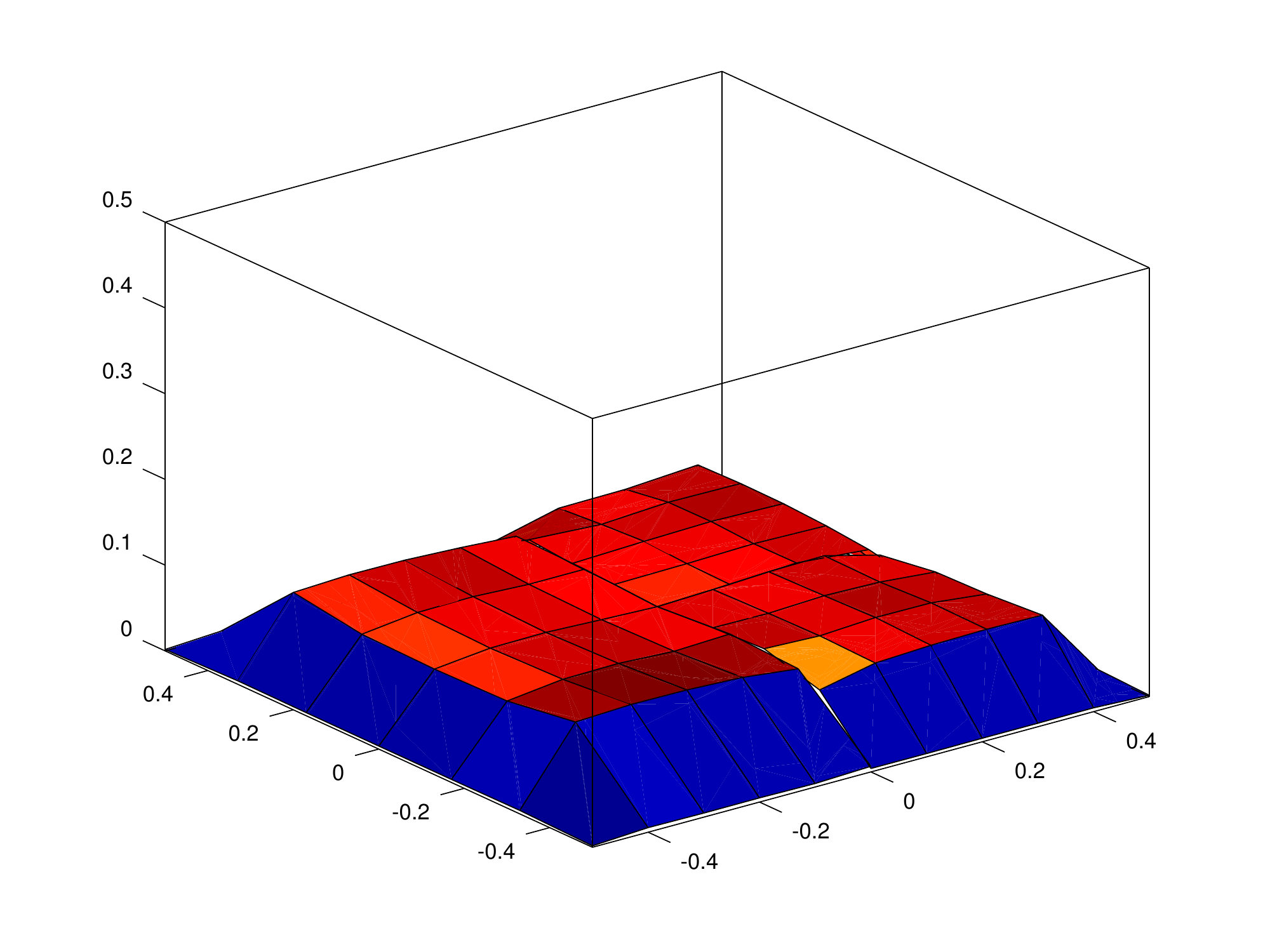}
\includegraphics[width=0.45\textwidth]{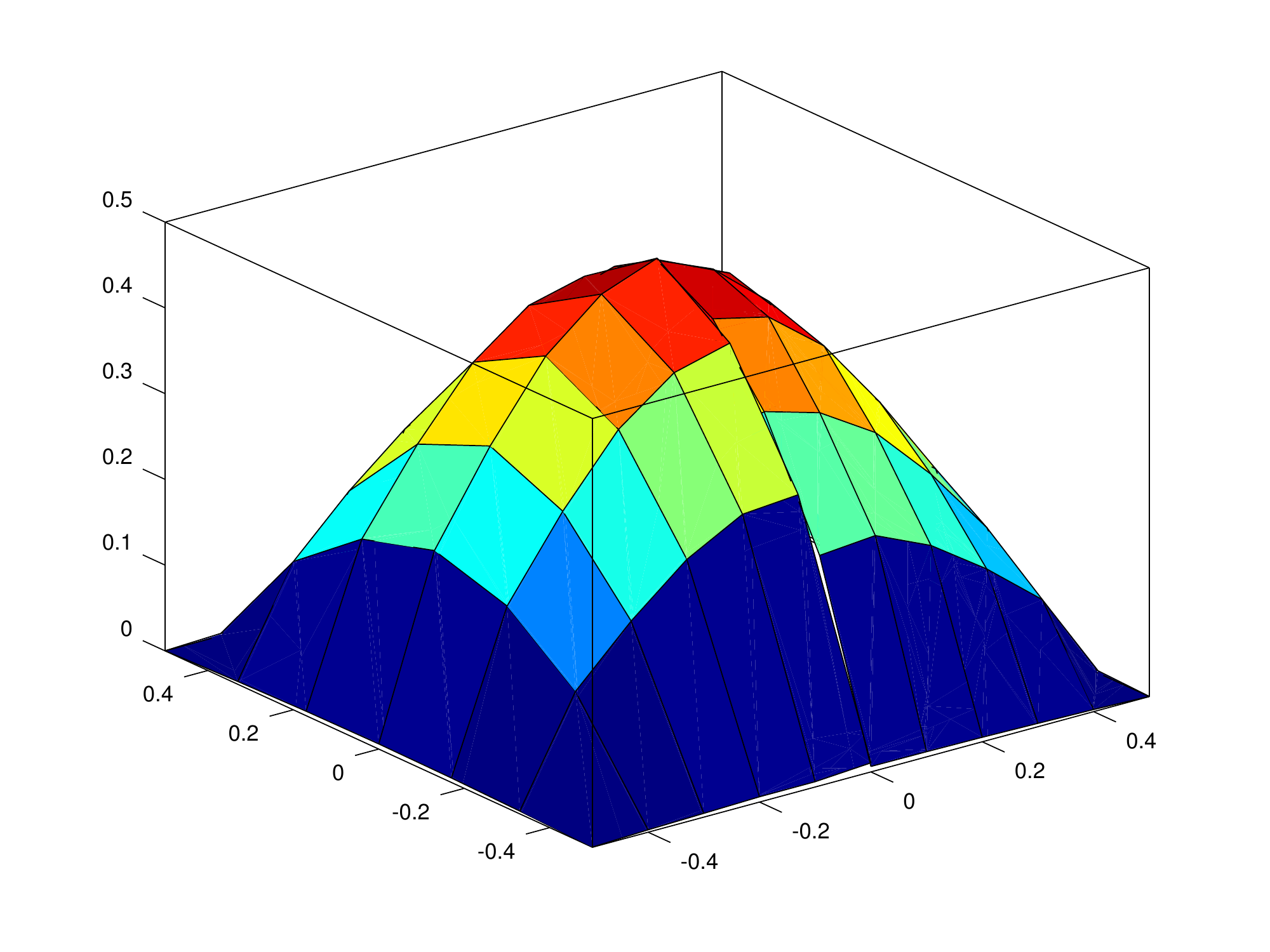}

\includegraphics[width=0.45\textwidth]{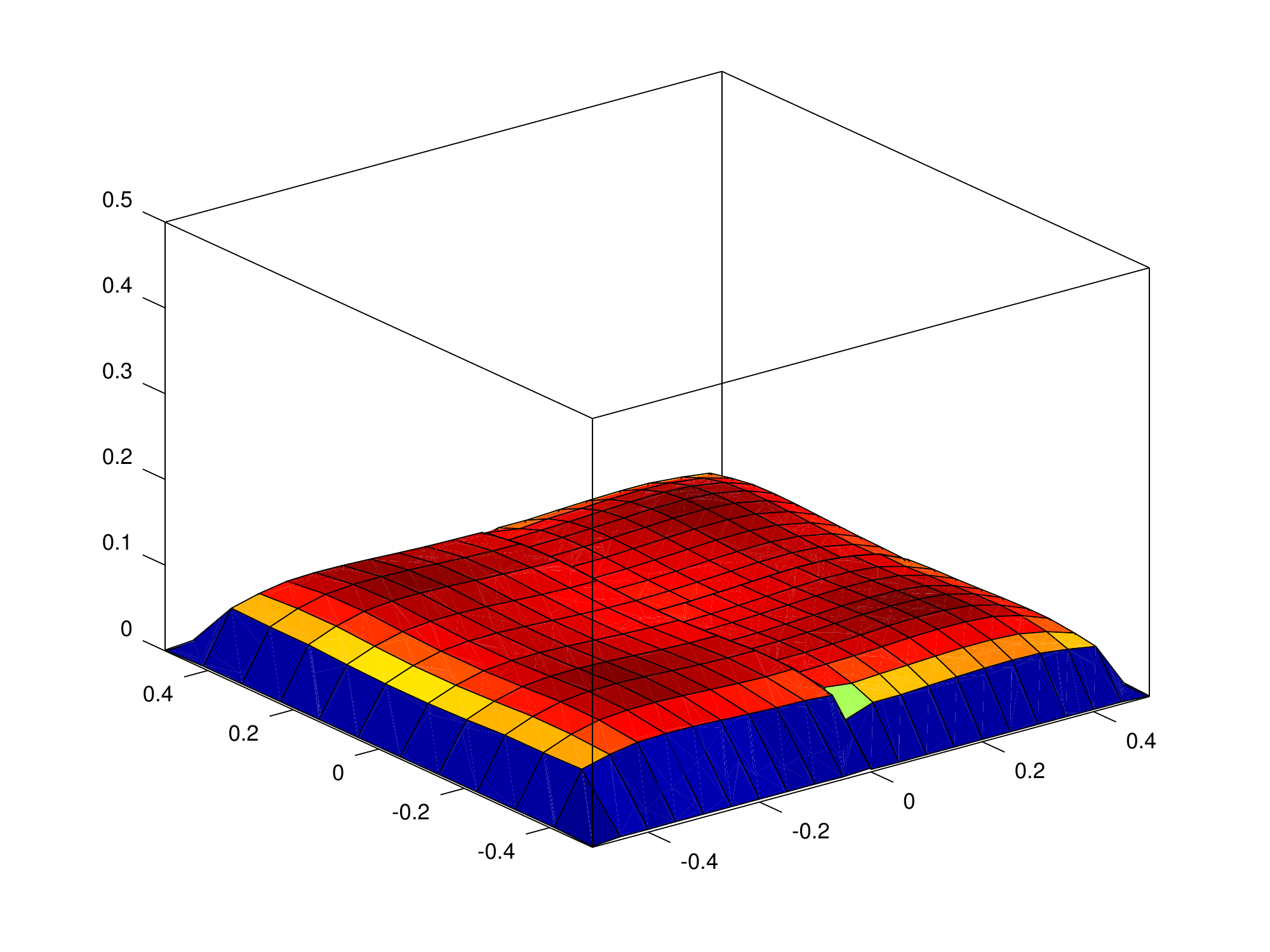}
\includegraphics[width=0.45\textwidth]{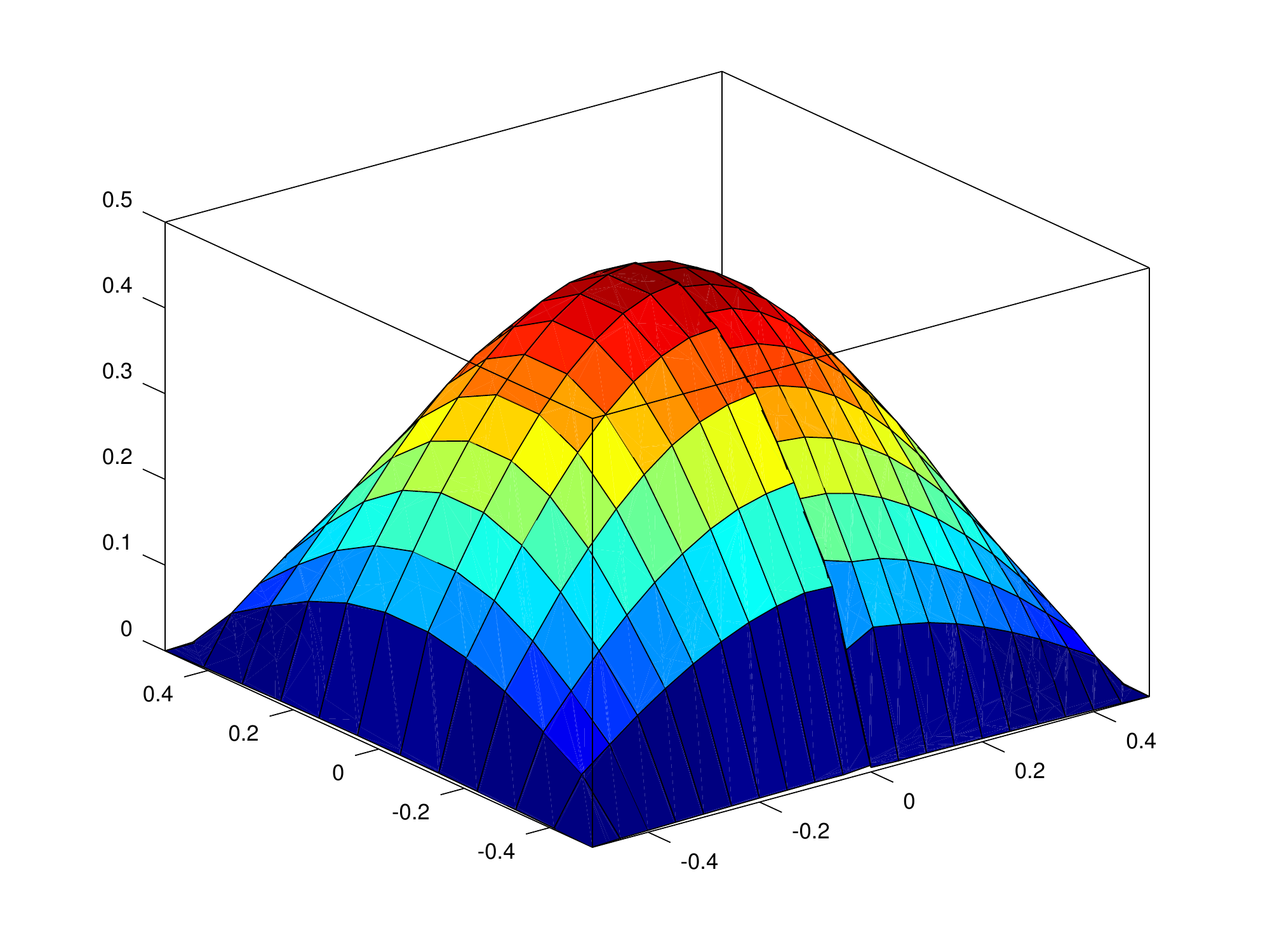}
\end{center}
\caption{Nitsche approximations ($k=5$, $\nu=10$),
         real parts (left) and imaginary parts (right).}
\label{fig_approx_Nitsche}
\end{figure}

\clearpage

\bibliographystyle{siam}
\bibliography{/home/norbert/tex/bib/bib,/home/norbert/tex/bib/heuer,/home/norbert/tex/bib/fem}

\end{document}